\documentclass{amsart}

\usepackage{amsmath,amssymb,amsthm,latexsym,color,geometry,rotating}

\usepackage{bm,natbib}

\theoremstyle{plain}

\newtheorem{prop}{Proposition}

\theoremstyle{definition}
\newtheorem{remark}{Remark}

\usepackage{ifthen}
\newcommand*{\equ}[2][\empty]{
	\begin{equation}
	\ifthenelse{\equal{#1}{\empty}}{}{\label{#1}}
	#2
	\end{equation}
}

\newcommand*{\E}{\mbox{E}}

\begin{document}


\title{Bias-correction of the maximum likelihood estimator for the $\alpha$-Brownian bridge}

\author{Maik G\"orgens}
\author{M\aa ns Thulin}
\date{\today}

\begin{abstract}
\noindent 
The $\alpha$-Brownian bridge, or scaled Brownian bridge, is a generalization of the Brownian bridge with a scaling parameter that determines how strong the force that pulls the process back to 0 is. The bias of the maximum likelihood estimator of the parameter $\alpha$ is derived and a bias-correction that improves the estimator substantially is proposed. The properties of the bias-corrected estimator and four Bayesian estimators based on non-informative priors are evaluated in a simulation study.
\noindent 
   \\[1.5mm] {\bf Keywords:} $\alpha$-Brownian bridge, bias-correction, estimation, scaled Brownian bridge.
\end{abstract}

\maketitle


\section{Introduction}

Let $W=(W_t)_{t\in\lbrack 0,1\rbrack}$ be a standard Brownian motion. For $\alpha\geq 0$, consider the stochastic differential equation $$dX_t^{(\alpha)}=dW_t-\frac{\alpha}{1-t}X_t^{(\alpha)}dt, \quad X_0^{(\alpha)}=0, \quad t\in \lbrack 0,1).$$ The solution to this equation is $X^{(\alpha)}=(X_t^{(\alpha)})_{t\in\lbrack 0,1)}$ with $$X_t^{(\alpha)}=\int_0^t\Big(\frac{1-t}{1-s}\Big)^\alpha dW_s,\qquad t\in\lbrack 0,1),$$
a process which returns to 0 at time 1 almost surely when $\alpha>0$. The process $X^{(\alpha)}$ is known as an $\alpha$-Brownian bridge, $\alpha$-Wiener bridge or scaled Brownian bridge and can be viewed as a flexible alternative to the standard Brownian bridge. It includes Brownian motion ($\alpha=0$) and the usual Brownian bridge ($\alpha=1)$ as special cases. This paper is concerned with estimation of the scaling parameter $\alpha$ given a sample path observed until time $T<1$. Estimating $\alpha$ is of interest since $\alpha$ determines \emph{how} the process tends to 0: if $0<\alpha<1$ the process tends to 0 slower than the Brownian bridge and if $\alpha>1$ it tends to 0 faster than the Brownian bridge (see Figure~\ref{fig1}).
\begin{figure}
\begin{center}
 \caption{The influence of $\alpha$ on the ``expected future'' for different values of $\alpha$.}\label{fig1}
   \includegraphics[width=\textwidth]{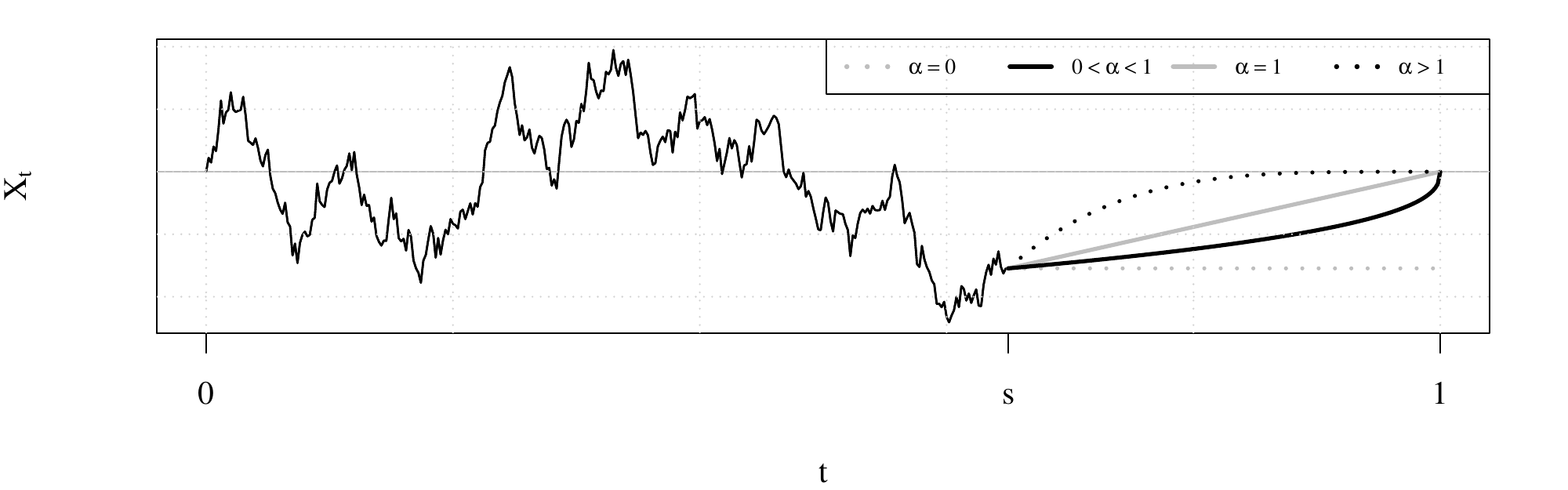}
 \end{center}
\end{figure}

The $\alpha$-Brownian bridge was introduced by \citet{brennan90}, who used it to model arbitrage profit in the absence of transaction costs. It has also been used to model exchange rate dynamics \citep{trede07} and the degree of interventionist policies in the run-up to a monetary union \citep{sondermann11}. A possible further application in financial theory is the following. It has been reported that stock prices tend to end up at strike prices of heavily traded vanilla options at the time of their maturity; see for example~\citet{Ave03} and the references therein. Brownian bridges have been used to model this phenomenon. However, the rate of convergence to the strike price will likely depend on the behavior of the market. In a cautious market, traders will start early on to push the stock price to the strike price. Then an $\alpha$-Brownian bridge with an $\alpha > 1$ will be more suitable to model the pinning behavior than the usual Brownian bridge. In an incautious  market, traders will start late to push the stock price to the strike price and an $\alpha$-Brownian bridge with an $0 < \alpha < 1$ will be more suitable to model the behavior. But the correct value of $\alpha$ is required in order to develop optimal selling strategies for stocks showing the pinning phenomenon. In this context we mention that the optimal stopping problem for the usual Brownian bridge was solved in~\citet{Eks09}. Other possible areas of applications include for instance modeling of animal movements, where the $\alpha$-Brownian bridge seems like a strong candidate to replace the Brownian bridge model proposed by \citet{horne07}.

Starting with a paper by \citet{mansuy04}, the $\alpha$-Brownian bridge has attracted considerable interest in the stochastics community. Estimation was discussed by \citet{barczy10}, \citet{barczy11a} and \citet{zhao12}, all of whom studied some properties of the maximum likelihood estimator (MLE) of $\alpha$. \citet{sebaiy13} studied the least squares estimator of $\alpha$ when $W$ is replaced by a fractional Brownian motion. \citet{zhao13} and \citet{gorgens14} studied hypothesis testing problems for $\alpha$.

Previous studies of estimators of $\alpha$ have focused on asymptotic properties, i.e. the behavior of the MLE as $T\rightarrow 1$. In Section \ref{bias} we derive the bias of the MLE for $T<1$, which is found to be surprisingly large. In Section \ref{estimators} we propose a bias-correction of the MLE and introduce some Bayesian estimators of $\alpha$. We then evaluate the properties of the Bayesian estimators and the bias-corrected MLE in a simulation study. Finally, some open problems are discussed in Section \ref{discussion} and  proofs are given in an appendix.

\begin{remark}
We mention here that we may just as well define the $\alpha$-Brownian bridge on an interval $[0,S]$: let $X^{(\alpha, S)}=(X^{(\alpha, S)}_t)_{t \in [0,S)}$ be the strong solution of the stochastic differential equation
\[ dX^{(\alpha, S)}_t = dW_t - \frac{\alpha X^{(\alpha, S)}_t}{S-t} dt, \qquad X^{(\alpha, S)}_0 = 0, \quad 0 \leq t < S. \]
Then, for $\alpha > 0$, we have $\lim_{t \rightarrow S} X^{(\alpha, S)}_t = 0$. The $\alpha$-Brownian bridge is self-similar:
\[ \left( X^{(\alpha, S)}_t \right)_{t \in [0, S]} =_d \left( \sqrt{S} X^{(\alpha, 1)}_{t/S} \right)_{t \in [0, S]}. \]
From this self-similarity the results in this paper easily extend to $\alpha$-Brownian bridges on an interval $[0,S]$, so it suffices to study the simpler setting where $S=1$.
\end{remark}

\section{The bias of the MLE}\label{bias}

The MLE of $\alpha$ based on a trajectory observed until time $T < 1$ is given by
\[
\hat{\alpha}_{MLE}=-\Big(\int_0^t\frac{X_s^{(\alpha)}}{1-s}dX_s^{(\alpha)}\Big)\Big/\Big(\int_0^t\frac{(X_s^{(\alpha)})^2}{(1-s)^2}ds\Big)\qquad\mbox{for } t\in (0,1).
\]
\citet{barczy11b} showed that $\hat{\alpha}_{MLE}$ is a strongly consistent estimator of $\alpha$. It is also the least squares estimator \citep{sebaiy13}. \citet[Section 2.2]{gorgens14} showed that the MLE can be written as
\begin{equation}\label{mle.eq}
\hat{\alpha}_{MLE}=\Big(-\frac{(X_T^{(\alpha)})^2}{1-T}+\int_0^T\frac{(X_s^{(\alpha)})^2}{(1-s)^2}ds-\ln(1-T)\Big)\Big/\Big(2\int_0^T\frac{(X_s^{(\alpha)})^2}{(1-s)^2}ds\Big)
\end{equation}
which allows for a straightforward computation of $\hat{\alpha}_{MLE}$ given the sample path.

Next, we present a result for the expected value of the MLE, the proof of which is given in the appendix.

\begin{prop}\label{prop1}
  The expected value of the maximum likelihood estimator $\hat{\alpha}_{MLE}$ of $\alpha$ based on the observations up to time $T<1$ is
  \begin{align}
    \E_\alpha [ \hat{\alpha}_{MLE} ] &= \frac{1}{2} + \frac{(1-T)^{\frac{1-2\alpha}{4}} }{\sqrt{2}} \int_{\frac{1-2\alpha}{2}}^\infty \frac{ (1-T)^u - (1-T)^{-u} }{\left((1 + \frac{1-2\alpha}{2u})(1-T)^u + (1 - \frac{1-2\alpha}{2u})(1-T)^{-u}\right)^{3/2}} du \label{E:Gen_Form} \\
    &\qquad - \frac{\ln(1-T)(1-T)^{\frac{1-2\alpha}{4}} }{\sqrt{2}} \int_{\frac{1-2\alpha}{2}}^\infty \frac{u}{\sqrt{(1 + \frac{1-2\alpha}{2u})(1-T)^u + (1 - \frac{1-2\alpha}{2u})(1-T)^{-u}}} du. \notag
  \end{align}
\end{prop}

The expectation ~\eqref{E:Gen_Form} and the corresponding bias are shown in Figure \ref{fig2} for $T\in\{0.7,0.8,0.9\}$ and $\alpha\in\lbrack 0,10\rbrack$. The bias follows the same pattern for different $T$, and is increasing in $T$. We can see an almost constant behavior of the bias of $\hat{\alpha}_{MLE}$ for larger values of $\alpha$. The following result gives an explanation of this. We use the notation $f(\alpha) \sim g(\alpha)$ for $\lim_{\alpha \rightarrow \infty} [f(\alpha) - g(\alpha)] = 0$.
\begin{figure}
\begin{center}
 \caption{The expectation and the bias of the MLE when $T\in\{0.7,0.8,0.9\}$.}\label{fig2}
   \includegraphics[width=\textwidth]{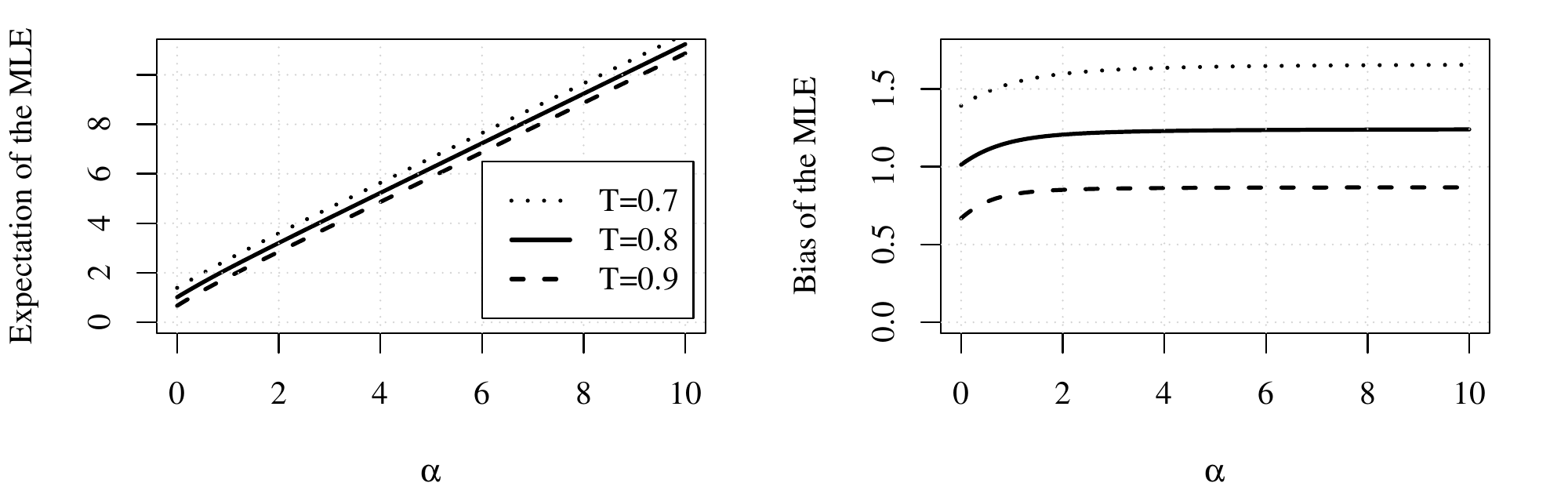}
 \end{center}
\end{figure}

\begin{prop}\label{prop2}
  For the bias $\E_\alpha [ \hat{\alpha}_{MLE} ]$ of the MLE based on the observations up to time $T<1$ we have
  \[ \E_\alpha [ \hat{\alpha}_{MLE} ] - \alpha \sim - \frac{2}{\ln(1-T)}. \]
\end{prop}
A proof is found in the appendix.

When $\alpha=1/2$ we can obtain a simpler expression for the bias. Specializing formula~\eqref{E:Gen_Form} to the case $\alpha=1/2$ we obtain
\[ \E_\alpha [ \hat{\alpha}_{MLE} ] = \frac{1}{2} + \frac{1}{\sqrt{2}} \int_0^\infty \frac{ (1-T)^u - (1-T)^{-u} }{\left((1-T)^u + (1-T)^{-u}\right)^{3/2}} du - \frac{\ln(1-T)}{\sqrt{2}} \int_0^\infty \frac{u}{\sqrt{(1-T)^u + (1-T)^{-u}}} du. \]
The substitution $v = \ln(1-T)^u$ in both integrals gives
\equ[E:spec]{ \E_\alpha [ \hat{\alpha}_{MLE} ] = \frac{1}{2} + \frac{1}{\sqrt{2} \ln(1-T)} \int_0^\infty \frac{ e^v - e^{-v} }{\left(e^v + e^{-v}\right)^{3/2}} dv - \frac{1}{\sqrt{2}\ln(1-T)} \int_0^\infty \frac{v}{\sqrt{e^v + e^{-v}}} dv. }
The first integral yields
\equ[E:Int5]{ \int_0^\infty \frac{ e^v - e^{-v} }{\left(e^v + e^{-v}\right)^{3/2}} dv = \left[ - \left(e^v + e^{-v}\right)^{-1/2} \right]_0^\infty = \sqrt{2}, }
and the second integral may be rewritten as
\equ[E:Int6]{ \int_0^\infty \frac{v}{\sqrt{e^v + e^{-v}}} dv = \frac{1}{\sqrt{2}} \int_0^\infty \frac{v}{\sqrt{\cosh(v)}} dv = \frac{1}{\sqrt{2}} A, }
where
\[ A := \int_0^\infty \frac{v dv}{\sqrt{\cosh(v)}} \approx 5.5629. \]
Plugging~\eqref{E:Int5} and~\eqref{E:Int6} into~\eqref{E:spec} we obtain
\[ \E_{1/2} [ \hat{\alpha}_{MLE} ] = \frac{1}{2} + \frac{1 - A/2}{\ln(1-T)}. \]
The bias of the MLE when $\alpha=1/2$ and $0.5<T<1$ is shown in Figure \ref{fig3}. As can be seen, it is quite substantial unless $T$ is very close to 1. In particular, $\E_{1/2} [ \hat{\alpha}_{MLE} ]>1$ when $T<0.97$, meaning that on average the incautious market with $\alpha=1/2$ will be mistaken for a cautious market with $\alpha>1$.
\begin{figure}
\begin{center}
 \caption{The bias of the MLE when $\alpha=1/2$ and $0.5<T<1$.}\label{fig3}
   \includegraphics[width=\textwidth]{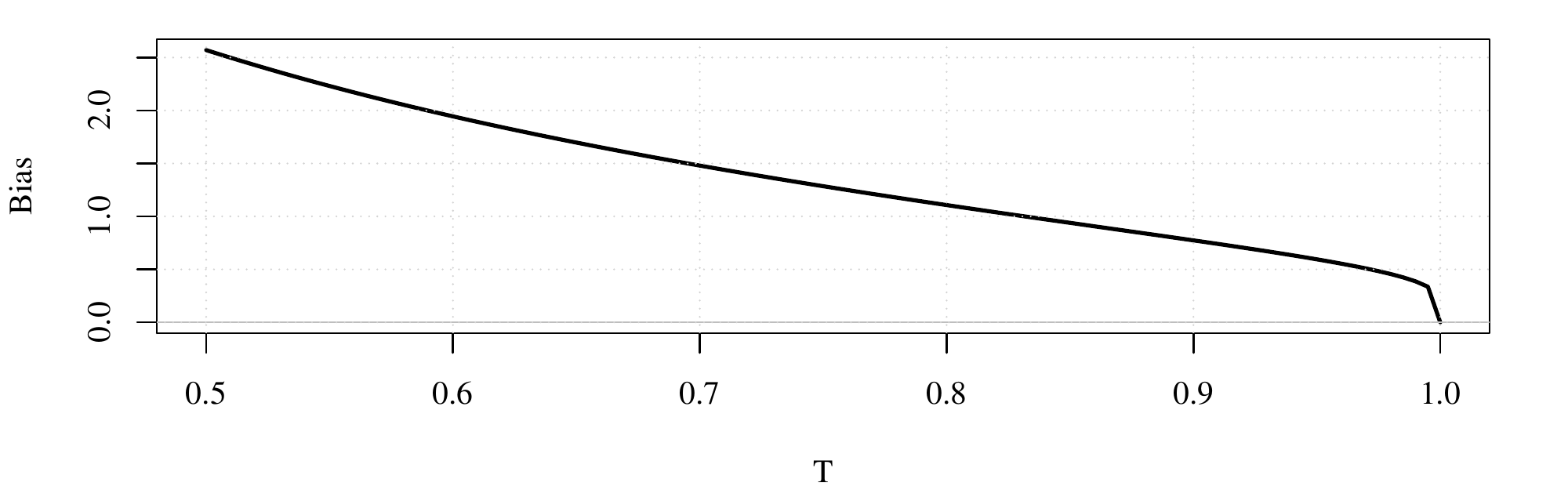}
 \end{center}
\end{figure}

\section{Long-run performance of other estimators}\label{estimators}
\subsection{Alternatives to the MLE}
A bias-correction of the MLE can be obtained by inverting the expectation ~\eqref{E:Gen_Form} numerically, letting the corrected estimator be
$$\hat{\alpha}_{CMLE}=\E_\alpha^{-1} [ \hat{\alpha}_{MLE} ],$$
so that $\hat{\alpha}_{CMLE}$ is the $\alpha$ for which $\E_\alpha [ \hat{\alpha}_{MLE}]$ is the observed value of $\hat{\alpha}_{MLE}$.

Other alternative estimators are Bayesian estimators of $\alpha$. We will study the mean and median of posterior distributions based on the Jeffreys and $U(0,10)$ priors. These can be viewed either as Bayesian estimators or regularized frequentist estimators.

The non-informative \citet{jeffreys46} prior $\pi_J(\alpha|T)$ is proportional to $\sqrt{I_\alpha(T)}$ where $I_\alpha(T)$ is the Fisher information. From \citet[Lemma 10]{barczy11b} it follows that
\[
\pi_J(\alpha|T)\propto \begin{cases}
    \frac{1}{2\alpha-1}\sqrt{(1-T)^{1/4-\alpha/2}-1-\ln(\lbrack 1-T\rbrack^{2\alpha-1})}, & \text{if $\alpha\neq 1/2$}.\\
    \frac{1}{\sqrt{2}}\ln(1-T), & \text{if $\alpha=1/2$}.
  \end{cases}
\]

The Jeffreys prior for $\alpha$ is very heavy-tailed. Its median is roughly 98 when $T=0.8$ and 94 when $T=0.95$, meaning that much of the prior probability mass is concentrated on values of $\alpha$ that seem very unlikely to occur in practice. In applications where no prior information is available, it might therefore be preferable to use a ``low-informative'' prior with bounded support, such as the $U(0,10)$ prior.

The posterior distributions are computed by Bayes formula using the likelihood function, which is
\begin{equation}\label{E:likelihood}
f(\alpha|(X_s)_0^T)=\exp\Big( -\frac{\alpha(X_T^{(\alpha)})^2}{2(1-T)}+\frac{\alpha(1-\alpha)}{2}\int_0^T\frac{(X_s^{(\alpha)})^2}{(1-s)^2}ds-\frac{\alpha}{2}\ln(1-T)  \Big),
\end{equation}
see \citet[Section 2.1]{gorgens14}.

\subsection{Simulation study}

To estimate the bias and MSE of the alternative estimators we performed a simulation study, in which for 100,000 realizations of the process were simulated $\alpha\in\{0,0.5,1,1.5,2,2.5,3,4,6,8,10\}$. Each realization was observed in 300 points and the estimators were computed by approximating the integrals in \eqref{mle.eq} and \eqref{E:likelihood} using the rectangle rule.

The bias and MSE of the alternative estimators are compared to that of the MLE when $T=0.8$ in Figure \ref{fig4}. The figure is qualitatively similar for other values of $T$.

From Figure \ref{fig4} we conclude that $\hat{\alpha}_{CMLE}$ is nearly unbiased and has a lower MSE than the MLE, thereby improving upon the MLE considerably. The Bayesian estimators based on the Jeffreys prior are biased, but except when $\alpha$ is close to 0 the bias is lower than that of the MLE. They also have lower MSE's, which is close to that of the corrected MLE. The Bayesian estimators based on the $U(0,10)$ prior shrinks the estimate towards 5 and have the best performance in that particular region of the parameter space. When $\alpha$ is close to 0 or 10 they are heavily biased. Among the five estimators, only the corrected MLE is nearly unbiased when $\alpha$ is close to 1, meaning that it is the only estimator that reliably identifies both cautious and incautious markets.

\begin{figure}
\begin{center}
 \caption{Bias and MSE of the MLE, the corrected MLE and four Bayesian estimators when $0\leq \alpha \leq 10$ and $T=0.8$.}\label{fig4}
   \includegraphics[width=\textwidth]{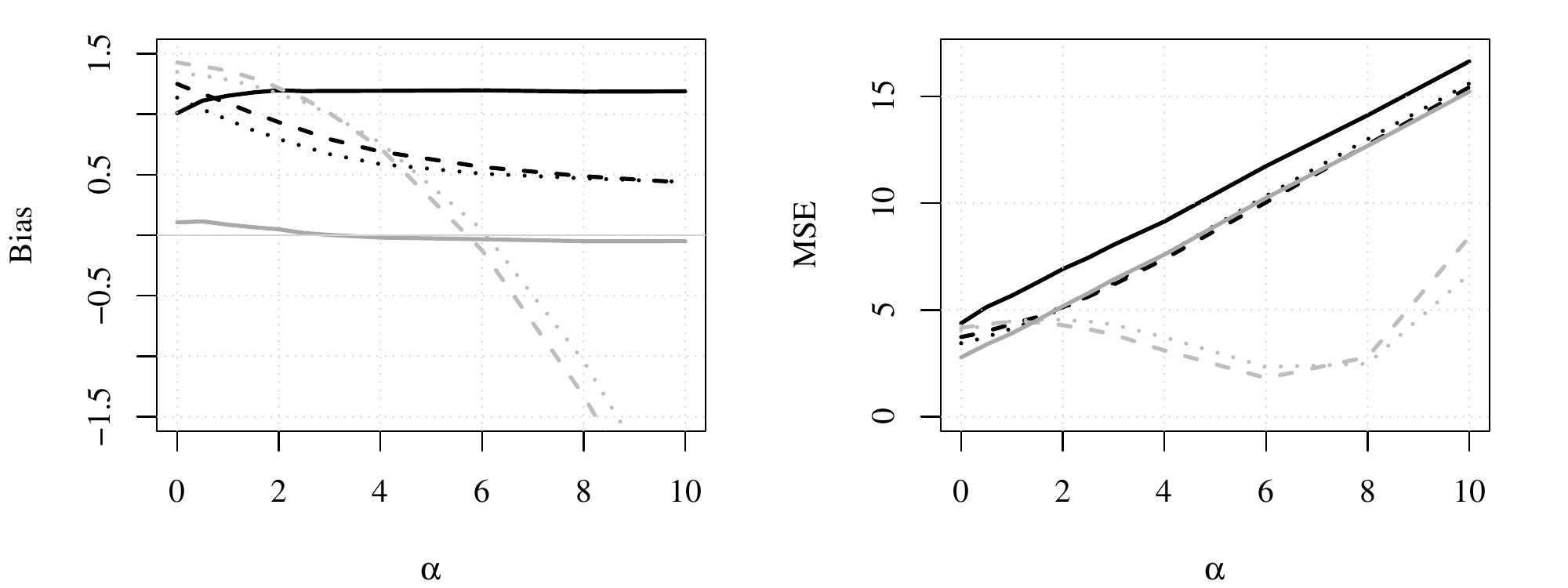}
    \includegraphics[width=\textwidth]{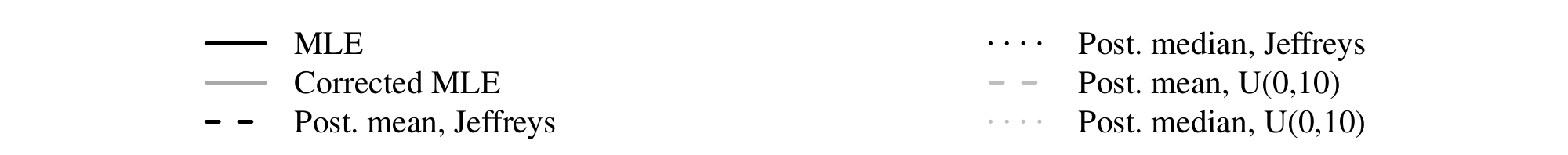}
 \end{center}
\end{figure}

\section{Discussion}\label{discussion}
We have shown both analytically and numerically that the MLE of $\alpha$ is heavily biased, but that it is possible to correct for this bias. The corrected estimator is nearly unbiased and has a lower MSE. It can be recommended for use instead of the MLE. If one for some reason is unwilling to apply the bias-correction, it might be preferable to use a Bayesian estimator based on the Jeffreys prior instead.

We have also seen that when the Jeffreys and $U(0,10)$ priors are used for Bayesian estimation of $\alpha$, the posterior mean and the posterior median often fail to identify incautious markets, that is, situations where $\alpha<1$. We note that although such long-run properties are of a frequentist nature, they still play a part in calibration of Bayesian procedures, for instance if one wishes to apply these estimators for repeated financial decisions. A bias-correction such as that applied to the MLE makes little sense in a Bayesian context, but there may be a truly Bayesian way to obtain estimators with better long-run properties. The estimators in our study are motivated by the squared error loss and the absolute value loss, respectively. For the $\alpha$-Brownian bridge it might prove fruitful to instead use a loss function that penalizes overestimation more than underestimation. We have however not pursued this idea further.

Although estimation and hypothesis testing for the $\alpha$-Brownian bridge has been studied extensively, interval estimation of $\alpha$ remains an open problem. \citet{barczy11b} derived the asymptotic distribution of the MLE, which can be used to construct a confidence interval, but based on our investigation of the properties of this estimator we doubt that the confidence interval based on the MLE will have good performance. Our investigation casts similar doubts on how well Bayesian credible sets based on the Jeffreys and $U(0,10)$ priors will perform. As confidence intervals are considerably more informative than point estimates, we believe this problem to be of great interest.

A further motivation for studying confidence intervals is their connection to hypothesis testing. Papers dealing with hypothesis tests for $\alpha$ \citep{zhao13,gorgens14} have focused on tests of point hypotheses of the type $H_0: \alpha=\alpha_0$ versus $H_1: \alpha=\alpha_1$. In many situations it would be of greater interest to test against a composite hypothesis, e.g. $H_0: \alpha=\alpha_0$ versus $H_1: \alpha\neq \alpha_0$ or $H_0: \alpha\leq\alpha_0$ versus $H_1: \alpha> \alpha_0$. This is possible to do by inverting a well-behaved confidence interval, not only in the frequentist setting, but also in Bayesian inference \citep{thulin14}.

\section*{Appendix}
In this appendix we give the proofs of Propositions \ref{prop1}-\ref{prop2}. Ignoring the superscript $(\alpha)$ and setting $I_T = \int_0^T \frac{(X^{(\alpha)}_s)^2}{(1-s)^2} ds$ we obtain the following simplified expression for the MLE.
\equ[E:MLE_Short]{ \hat{\alpha}_{MLE} = - \frac{X_T^2}{2 (1-T) I_T} + \frac{1}{2} - \frac{\ln(1-T)}{2 I_T}. }
From~\eqref{E:MLE_Short} we have
\equ[E:Exp_MLE]{ \E_{\alpha} \hat{\alpha}_{MLE} = - \frac{1}{2 (1-T)} \E_{\alpha} [X_T^2/I_T] + \frac{1}{2} - \frac{\ln(1-T)}{2} \E_{\alpha} [1/I_T], }
where $\E_\alpha$ denotes expectation under $\alpha$.

In order to compute $\E_{\alpha} [X_T^2/I_T]$ and $\E_{\alpha} [1/I_T]$ we will use the following
\begin{prop}[See~\citet{Wil41} and~\citet{Cre81}]\label{P:Gen_Quot}
  Let $Y$ and $Z$ be two positive random variables and let $M_{Y,Z}(s,t) = \E \exp(sY + tZ)$ be their joint Laplace transform. Then, for $j,k \geq 0$
  \equ[E:Gen_Quot]{ \E[Y^j/Z^k] = \Gamma(k)^{-1} \int_0^\infty t^{k-1} \lim_{s \searrow 0} \frac{\partial^j}{\partial s^j} M_{Y,Z}(s,-t) dt. }
\end{prop}

\begin{proof}[Proof of Proposition \ref{prop1}]
It was shown in Theorem~21 in~\citet{barczy11b} that the joint Laplace transform $M(s,t) = \E_{\alpha} \exp(sX_T^2 + tI_T)$ of $X_T^2$ and $I_T$ is
\[ M(s,t) = \frac{(1-T)^{(1-2\alpha)/4}}{\sqrt{\cosh( u(-t) \ln(1-T) ) + \frac{1-2\alpha+4s(1-T)}{2u(-t)} \sinh( u(-t) \ln(1-T) )}}, \]
where $u(t) = \sqrt{8t+(2\alpha-1)^2}/2$. Applying Proposition~\ref{P:Gen_Quot} with $Z = I_T$, $j=0$, and $k=1$ it follows that
\[ \E_{\alpha}[I_T^{-1}] = \int_0^\infty \frac{(1-T)^{(1-2\alpha)/4}}{\sqrt{\cosh ( u(t) \ln(1-T) ) + \frac{1-2\alpha}{2u(t)} \sinh( u(t) \ln(1-T) )}} dt. \]
The substitution $u = u(t)$ yields
\equ[E:Int1]{ \E_{\alpha}[I_T^{-1}] = \sqrt{2} (1-T)^{\frac{1-2\alpha }{4}} \int_{\frac{1-2\alpha }{2}}^\infty \frac{u du}{\sqrt{(1 + \frac{1-2\alpha }{2u})(1-T)^u + (1 - \frac{1-2\alpha}{2u})(1-T)^{-u}}}, }
where we also used
\begin{align*}
  \cosh(u \ln(1-T)) &= \frac{1}{2}((1-T)^u + (1-T)^{-u}),\qquad \text{and} \\
  \sinh(u \ln(1-T)) &= \frac{1}{2}((1-T)^u - (1-T)^{-u}).
\end{align*}

The derivative of $M(s,t)$ with respect to $s$ is given by
\[ \frac{\partial}{\partial s} M(s,t) = \frac{-\frac{(1-T)^{(5-2\alpha)/4} }{u(-t)} \sinh\left( u(-t) \ln(1-T) \right) }{\left(\cosh\left( u(-t) \ln(1-T) \right) + \frac{1-2\alpha+4s(1-T)}{2u(-t)} \sinh\left( u(-t) \ln(1-T) \right)\right)^{3/2}}. \]
Applying Proposition~\ref{P:Gen_Quot} once again with $Y=X_T^2$, $Z = I_T$, $j=1$, and $k=1$ it follows that
\[ \E_{\alpha}[X_T^2/I_T] = \int_0^\infty \frac{-\frac{(1-T)^{(5-2\alpha)/4} }{u(t)} \sinh\left( u(t) \ln(1-T) \right) }{\left( \cosh ( u(t) \ln(1-T) ) + \frac{1-2\alpha}{2u(t)} \sinh( u(t) \ln(1-T) ) \right)^{3/2}} dt. \]
Again, the substitution $u = u(t)$ yields
\equ[E:Int2]{ \E_{\alpha}[X_T^2/I_T] = - \sqrt{2} (1-T)^{\frac{5-2\alpha}{4}} \int_{\frac{1-2\alpha }{2}}^\infty \frac{ (1-T)^u - (1-T)^{-u} }{\left((1 + \frac{1-2\alpha }{2u})(1-T)^u + (1 - \frac{1-2\alpha}{2u})(1-T)^{-u}\right)^{3/2}} du. }

Plugging~\eqref{E:Int1} and~\eqref{E:Int2} into~\eqref{E:Exp_MLE} we obtain \eqref{E:Gen_Form}.
\end{proof}

\begin{remark}
  Using Proposition~\ref{P:Gen_Quot} we could also find a formula for the mean squared error of $\hat{\alpha}_{MLE}$ since
  \begin{align*}
    \E_\alpha[(\hat{\alpha}_{MLE} - \alpha)^2] &= (\alpha - 1/2)^2 + \frac{\alpha - 1/2}{1-T} \E_{\alpha}[X_T^2/I_T] + (\alpha - 1/2) \ln(1-T) \E_{\alpha}[ 1/I_T ] \\
      &\qquad  + \frac{1}{4(1-T)^2} \E_{\alpha}[X_T^4/I_T^2] + (\ln(1-T))^2 \E_{\alpha}[1/I_T^2] + \frac{\ln(1-T)}{2(1-T)} \E_{\alpha}[X_T^2/I_T^2].
  \end{align*}
  However, at this time we do not see a way to simplify the occurring integrals significantly and thus we do not pursue this further.
\end{remark}

\begin{proof}[Proof of Proposition \ref{prop2}]
  From~\eqref{E:Gen_Form} we know that
  \equ[E:Short_Hand]{ \E_\alpha [ \hat{\alpha}_{MLE} ] - \alpha = \frac{1}{2} - \alpha + I_1(\alpha, T) - I_2(\alpha, T), }
  where
\equ[E:Int1Def]{ I_1(\alpha, T) := \frac{(1-T)^{\frac{1-2\alpha}{4}} }{\sqrt{2}} \int_{\frac{1-2\alpha}{2}}^\infty \frac{ (1-T)^u - (1-T)^{-u} }{\left((1 + \frac{1-2\alpha}{2u})(1-T)^u + (1 - \frac{1-2\alpha}{2u})(1-T)^{-u}\right)^{3/2}} du }
and
\equ[E:Int2Def]{ I_2(\alpha, T) := \frac{\ln(1-T)(1-T)^{\frac{1-2\alpha}{4}} }{\sqrt{2}} \int_{\frac{1-2\alpha}{2}}^\infty \frac{u}{\sqrt{(1 + \frac{1-2\alpha}{2u})(1-T)^u + (1 - \frac{1-2\alpha}{2u})(1-T)^{-u}}} du. }

The substitution $v=2u/(2\alpha-1)$ in $I_1(\alpha, T)$ yields
\begin{align*}
  I_1(\alpha, T) &= \frac{(1-T)^{\frac{1-2\alpha}{4}} }{\sqrt{2}} \frac{2\alpha-1}{2} \int_1^\infty \frac{ (1-T)^{v(2\alpha-1)/2} - (1-T)^{-v(2\alpha-1)/2} }{\left((1 - \frac{1}{v})(1-T)^{v(2\alpha-1)/2} + (1 + \frac{1}{v})(1-T)^{-v(2\alpha-1)/2}\right)^{3/2}} dv \\
    &= \frac{2\alpha-1}{2\sqrt{2}} \int_1^\infty \frac{ (1-T)^{(v-1/2)(2\alpha-1)/2} - (1-T)^{(-v-1/2)(2\alpha-1)/2} }{\left((1 - \frac{1}{v})(1-T)^{v(2\alpha-1)/2} + (1 + \frac{1}{v})(1-T)^{-v(2\alpha-1)/2}\right)^{3/2}} dv \\
\end{align*}
The terms $(1-T)^{(v-1/2)(2\alpha-1)/2}$ in the nominator and $(1 - \frac{1}{v})(1-T)^{v(2\alpha-1)/2}$ in the denominator of the integrand vanish as $\alpha$ tends to infinity and thus
\begin{align*}
  I_1(\alpha, T) &\sim - \frac{2\alpha-1}{2\sqrt{2}} \int_1^\infty \frac{ (1-T)^{(-v-1/2)(2\alpha-1)/2} }{\left((1 + \frac{1}{v})(1-T)^{-v(2\alpha-1)/2}\right)^{3/2}} dv \\
    &= - \frac{2\alpha-1}{2\sqrt{2}} \int_1^\infty \left(1 + \frac{1}{v} \right)^{-3/2} (1-T)^{(v-1)(2\alpha-1)/4} dv.
\end{align*}
By partial integration we obtain
\begin{align*}
 I_1(\alpha, T) &\sim - \frac{2\alpha-1}{2\sqrt{2}} \left[ \frac{4 \left(1 + \frac{1}{v} \right)^{-3/2} (1-T)^{(v-1)(2\alpha-1)/4}}{ (2\alpha-1) \ln(1-T) } \right]_1^\infty \\
    &\qquad + \frac{2\alpha-1}{2\sqrt{2}} \int_1^\infty \frac{ 3 \cdot 4 (1-T)^{(v-1)(2\alpha-1)/4}}{ 2 \left(1 + \frac{1}{v} \right)^{5/2} v^2 (2\alpha-1) \ln(1-T) } dv \\
   &= \frac{1}{2 \ln(1-T)} + \frac{3}{\ln(1-T)} \int_1^\infty \frac{(1-T)^{(v-1)(2\alpha-1)/4}}{\left(1 + \frac{1}{v} \right)^{5/2} v^2 } dv.
\end{align*}
The latter integral vanishes for large $\alpha$ and thus
\equ[E:Int1Final]{ I_1(\alpha, T) \sim \frac{1}{2 \ln(1-T)}. }

For the integral $I_2(\alpha, T)$ we proceed in a similar way. The substitution $v=2u/(2\alpha-1)$ yields
\[ I_2(\alpha, T) = \frac{\ln(1-T)}{\sqrt{2}} \left(\frac{2\alpha-1}{2}\right)^2 \int_1^\infty \frac{v dv}{\sqrt{(1 - \frac{1}{v})(1-T)^{(v+1)(2\alpha-1)/2} + (1 + \frac{1}{v})(1-T)^{(1-v)(2\alpha-1)/2}}}.  \]
The term $(1 - \frac{1}{v})(1-T)^{(v+1)(2\alpha-1)/2}$ vanishes as $\alpha$ tends to infinity and thus
\[ I_2(\alpha, T) \sim \frac{\ln(1-T)}{\sqrt{2}} \left(\frac{2\alpha-1}{2}\right)^2 \int_1^\infty \left( \frac{v^3}{v+1} \right)^{1/2} (1-T)^{(v-1)(2\alpha-1)/4} dv. \]
Partial integration yields
\begin{align*}
 I_2(\alpha, T) &\sim \frac{\ln(1-T)}{\sqrt{2}} \left(\frac{2\alpha-1}{2}\right)^2 \left[ \frac{4 \left( \frac{v^3}{v+1} \right)^{1/2} (1-T)^{(v-1)(2\alpha-1)/4}}{ (2\alpha-1) \ln(1-T) } \right]_1^\infty \\
    &\qquad- \frac{\ln(1-T)}{\sqrt{2}} \left(\frac{2\alpha-1}{2}\right)^2 \int_1^\infty \frac{4 v^2(2v+3) (1-T)^{(v-1)(2\alpha-1)/4}}{2\left( \frac{v^3}{v+1} \right)^{1/2} (v+1)^2 (2\alpha-1) \ln(1-T) } dv \\
   &= - \alpha + \frac{1}{2} - \frac{2\alpha-1}{2 \sqrt{2}} \int_1^\infty \frac{\sqrt{v} (2v+3)}{(v+1)^{3/2}} (1-T)^{(v-1)(2\alpha-1)/4} dv.
\end{align*}
Integrating by parts once again gives
\begin{align*}
 I_2(\alpha, T) &\sim - \alpha + \frac{1}{2} - \frac{2\alpha-1}{2 \sqrt{2}} \left[ \frac{4 \sqrt{v} (2v+3) (1-T)^{(v-1)(2\alpha-1)/4}}{ (v+1)^{3/2} (2\alpha-1) \ln(1-T) } \right]_1^\infty \\
    &\qquad+ \frac{2\alpha-1}{2 \sqrt{2}} \int_1^\infty \frac{3 \cdot 4 (1-T)^{(v-1)(2\alpha-1)/4}}{2 \sqrt{v} (v+1)^{5/2} (2\alpha-1) \ln(1-T) } dv  \\
   &= - \alpha + \frac{1}{2} + \frac{5}{2 \ln(1-T)} + \frac{3}{\sqrt{2} \ln(1-T)} \int_1^\infty \frac{(1-T)^{(v-1)(2\alpha-1)/4}}{\sqrt{v} (v+1)^{5/2}} dv.
\end{align*}
The latter integral vanishes for large $\alpha$ and thus
\equ[E:Int2Final]{ I_2(\alpha, T) \sim - \alpha + \frac{1}{2} + \frac{5}{2 \ln(1-T)}. }

By plugging~\eqref{E:Int1Final} and~\eqref{E:Int2Final} into~\eqref{E:Short_Hand} we obtain the result
\[ \E_\alpha [ \hat{\alpha}_{MLE} ] - \alpha \sim \frac{1}{2} - \alpha + \frac{1}{2 \ln(1-T)} + \alpha - \frac{1}{2} - \frac{5}{2 \ln(1-T)} = - \frac{2}{\ln(1-T)}. \qedhere \]
\end{proof}


\end{document}